\documentclass[12pt]{article}

\usepackage[utf8]{inputenc}
\usepackage{amsfonts}
\usepackage{amssymb}
\usepackage{amsmath}
\usepackage{amsthm}
\usepackage{fullpage}
\usepackage{epsfig}
\usepackage{color}

\theoremstyle{plain}
\newtheorem{theorem}{Theorem}[section]

\newtheorem{corollary}[theorem]{Corollary}
\newtheorem{lemma}[theorem]{Lemma}
\theoremstyle{definition}

\newtheorem{assumption}[theorem]{Assumption}
\newtheorem*{example}{Example}

\theoremstyle{remark}
\newtheorem*{remark}{Remark}
\newtheoremstyle{cond}
  {3pt}			
  {3pt}			
  {}			
  {}			
  {\bfseries}		
  {.}			
  {.5em}		
  {Condition \thmnote{#3}}	
\theoremstyle{cond}

\def\theenumi{(\roman{enumi})}

\newcommand{\R}{\mathbb{R}}	
\newcommand{\C}{\mathbb{C} }    

\newcommand{\lip}{\textrm{Lip(X)}}
\newcommand{\lipX}{\textrm{Lip(X;$\:\R$)}}
\newcommand{\lipXC}{\textrm{Lip(X;$\:\C$)}}
\newcommand{\CC}{\mathcal{C}}
\newcommand{\esp}{\mathbb{E}}
\newcommand{\LL}{\mathcal{L}}
\newcommand{\FF}{\mathcal{F}}

\newcommand{\MZ}{\setminus\{0\}} 
\newcommand{\PP}{\mathcal{P}}
\newcommand{\CS}{\mathcal{S}}

\newcommand{\A}{\mathcal{A} }	

\begin{document}

\title{An explicit Berry-Esséen bound for uniformly expanding maps on the interval}
\author{Loïc Dubois\footnote{The last part of this paper was done in the University of Helsinki, and was partially funded by
The European Research Council}\\ \\
	Department of Mathematics,\\
	University of Cergy-Pontoise,\\
	2 avenue Adolphe Chauvin,\\
	95302 Cergy-Pontoise Cedex, France.}
\date{\today}
\maketitle

\abstract{ For uniformly expanding maps on the interval, analogous versions of the Berry-Esséen theorem are known but only with an
unexplicit upper bound in $O(1/\sqrt{n})$ without any constants being specified.
 In this paper, we use the recent complex cone technique to prove an
explicit Berry-Esséen estimate with a reasonable constant for these maps. Our method is not limited to maps on the interval however and
should apply to many situations.
}

\section{Introduction}
Let $(X_n)_{n\geq 1}$ be a sequence of independent, identically distributed (iid) real random variables.
Assume $\esp[X_k]=0$, $\esp[X_k^2]=\sigma^2>0$ and $\esp[|X_k|^3]=\rho<\infty$ then
the Berry-Essen theorem (see for instance \cite{Fe71}) claims that
\begin{equation} \label{BE-iid}
\left| P\left(\frac{X_1+\dots+X_n}{\sigma\sqrt{n}}\leq x\right)
-\frac{1}{\sqrt{2\pi}}\int_{-\infty}^x e^{-t^2/2}dt\right|
\leq \frac{3\rho}{\sigma^3\sqrt{n}}.
\end{equation}
Thus, for iid sequences, not only we know the speed of convergence in the central
limit theorem, but we have also a very precise
bound, which makes possible practical estimates by the normal law.

For deterministic systems, the situation is not so simple. If the system under consideration enjoys
sufficient decay of correlations for some class of observables (usually lipschitz or of bounded variations
functions), then one can prove a central limit theorem along with an analogous version of
the Berry-Esséen theorem. However, one does not get such a nice bound as
(\ref{BE-iid}) but only a $O(1/\sqrt{n})$ without the implied constant being specified.

In the present paper, we prove an explicit Berry-Esséen bound with a
reasonable constant for uniformly expanding Markov transformations on the interval and for
lipschitz observables. The novelty here is in the word `explicit'. The central limit theorem for such
transformations and for bounded variations observables was studied in (\cite{Wong79}, \cite{Ke80}).
In \cite{RE83} (see also \cite{Br96}), a Berry-Esséen theorem
is proved but with a non-explicit $O(1/\sqrt{n})$ bound.
The Berry-Esséen theorem for shifts of finite type was studied in \cite{CP90} but again
without any explicit rate of convergence. The determination of a reasonable constant in Theorem 1
of \cite{CP90} was actually left as an open problem. Though the formulation of
Theorem 1 of Parry and Coelho (\cite{CP90})
is a little bit different from the Berry-Essen estimate we prove in Theorem
\ref{main-thm}, one can easily use our calculations to give an explicit constant in their theorem,
see Remark \ref{remark-PC90}. This is essentially a matter of presentation.

More precisely, we prove the following theorem. We consider the probability $P$ given by the Gibbs measure $m_0$
associated to the uniformly expanding map $T$. We assume that  $T$ satisfies some Markov
condition, namely that each inverse branch of $T$ is defined on $[0,1]$.
We denote also $\gamma=\inf |T'|>1$. The observable $f$ satisfies $\esp[f]=0$ and is supposed to be lipschitz.
Finally, we denote $\sigma^2=\lim(1/n)\esp[(S_n f)^2]$ where $S_n f =\sum_{k=0}^{n-1} f\circ T^k$.
See section \ref{section-notations} for more details on the setting. The constant $G$ below depends only on $T$.
\begin{theorem} \label{main-thm}
Assume that $\sigma>0$. Then we have for all $x\in\R$, all $n\geq 1$,
\begin{eqnarray*}
\Big| P\left(\frac{S_n f}{\sigma \sqrt{n}} \leq x\right) &-& \frac{1}{\sqrt{2\pi}}\int_{-\infty}^x
e^{-t^2/2}dt\Big| \\
&\leq& C\frac{\cosh^6(D_\R /4) \|f\|_\infty \left(\|f\|_\infty +|f|_\ell\right)^2}
{\sigma^3 \sqrt{n}}.
\end{eqnarray*}
In the preceding inequality, $C$ is a numerical constant (one may take $C=11460$).
The constant $D_\R$ depends only on the dynamic and can be taken to be
\begin{equation} \label{eq-DR}
 D_\R = \frac{2(\gamma^2G +1)}{\gamma(\gamma-1)} + 2\log\frac{2\gamma^2G+\gamma +1}{\gamma -1}.
\end{equation}
\end{theorem}

I do not claim that this bound is optimal in any way. In fact, several choices in the proof of
Theorem \ref{main-thm} are a compromize to get a not too complicated formula. However, I believe that
to improve significantly this estimate, one has to improve the method. This can be seen by considering
the `size' of the last term in (\ref{eq-feller}).

The theorem is also valid if $T$ is an expanding map on any compact metric space (see section
\ref{section-notations}). The assumption that $T$ is Markov is quite strong and is
not strictly necessary. The method works as soon as one can find an (explicit) real cone which is
contracted by the transfer operator. It thus is possible to extend our result to more general expanding transformations
on the interval using the same cones as in \cite{Liv95bis}. However, this does not provide a completely
explicit bound (though constructible), and the formulas become quite complicated. We indicate in
section \ref{appendix-nonmarkov} how to extend our result to non-markov situations.

Our approach in this paper is similar to the spectral methods of \cite{CP90},
which rely on the equality
$$ \esp[\exp(it \sqrt{n}^{-1} S_n f)] =  \esp[ \LL(it/\sqrt{n})^n 1 ], $$
where $\LL(z)$ is a complex perturbation of the transfer operator $\LL$ (here normalized to have
$\LL 1= 1$).
The main difference--which allows to give an explicit bound--%
is that we replace standard perturbation theory of the spectrum (as in \cite{Kato80}) 
by the recent complex cones technique of Rugh (\cite{Rugh07}, \cite{Dub08}).
The idea is to compare the complex perturbation $\LL(z)$ with the positive operator $\LL=\LL(0)$. The operator
$\LL(0)$ contracts strictly a real cone with respect to the Hilbert metric, and under some conditions,
the operator $\LL(z)$ contracts strictly the complexification of this real cone with respect to a complex Hilbert metric.
This complex cone contraction gives better and simpler
bounds for both the size of the spectral gap of the perturbated operator $\LL(z)$ and the
size of the neighbourhood of $0$ in the complex plane on which the perturbated operator $\LL(z)$ has a spectral
gap. Since the leading eigenvalue of $\LL(z)$ along with its left and right eigenvectors depend
holomorphically on $z$, this is sufficient to get precise bounds.

Unfortunately, our method does not give easily explicit constants for the more refined estimates
of Parry and Coelho (\cite{CP90}).
Indeed, the Berry-Essen theorem only requires precise estimates of
the Taylor development of the Fourier transform $\esp[\exp(z S_n f)]$ --or equivalently, estimates
of the spectral gap of $\LL(z)$--for small
complex $z$. For further estimates, this is not enough, even in the independent case. One needs to know that
for all $t\in\R$ with $|t|\geq \delta>0$,
\begin{equation} \label{pb-refined}
\big|\esp[\exp(it S_n f)]\big| \xrightarrow[n\to\infty]{} 0.
\end{equation}
Of course, to get explicit constants for more refined estimates, one needs to have precise bounds for
the convergence in (\ref{pb-refined}).
In terms of the spectrum of $\LL(it)$, (\ref{pb-refined}) amounts to saying that the spectral radius of the
normalized transfer operator $\LL(it)$ is strictly less than $1$. This imposes conditions on the
observable. For subshifts of finite type, the spectral radius of the normalized transfer operator
$\LL(it)$ has spectral radius $1$ for some $t\neq 0$ if and only if the observable $f$ is
cohomologous to a continous function $l$ with values in $a+(2\pi/t)\mathbb{Z}$, or in other words, if and only
if there exists a continuous $\omega$ such that $f=l +\omega\circ T -\omega$ (see \cite{Po84}). If $f$
is not cohomologous to such a lattice valued function, then $f$ is called non-lattice.

The same problem arises for other kinds of limit theorems. For the local limit theorem (see for instance
\cite{RE83}), or for large deviation estimates, as soon as we know that $f$ is nonlattice, then one can apply
for instance the method of \cite{CS93} which gives strong large deviations; but if we want to explicit the
constants, one needs estimates for the convergence in (\ref{pb-refined}).

This paper is organized as follows. In Section \ref{section-ccone}, we briefly recall the necessary material
on complex cones. In Section \ref{sect-comp}, we prove Theorem \ref{thm-dh} and Theorem \ref{prop-comp2}.
Theorem \ref{prop-comp2} gives a general condition under which a complex operator `dominated' by a positive operator
is a complex cone contraction. Together with Theorem \ref{thm-dh}, it provides also an estimate of the
rate of contraction. These two theorems are actually direct extensions of Theorems 5.5 and 6.3 of \cite{Rugh07}.
The only additions --but essential here-- are the estimates of projective distances. It should be noticed however that the
original projective hyperbolic gauge in \cite{Rugh07} would lead (with additional work) to significantly worse estimates,
see Remark \ref{remark-estimates}.
The rest of the paper is devoted to the proof of Theorem \ref{main-thm}. In Section \ref{section-estdiam}, we develop the
dominated complex contraction argument in our situation, and finally, Section \ref{section-fourier} contains the proof
of the Berry-Esséen estimate.

Acknowledgment: the author expresses his deep thaks to Pr H.-H. Rugh for helping discussions during the preparation of this work.

\section{Notations} \label{section-notations}

Denote by $X=[0,1]$ the unit interval. We consider a metric $d$ on $X$ compatible with the topology of $X$ and
for which $X$ is of finite diameter at most $1$, ie $d(x,y)\leq 1$ for all $x$, $y$.
Denote by $\lipX$ (resp. $\lipXC$) the Banach algebra of all real
(resp. complex) valued bounded lipschitz functions on
$X$, endowed with the usual norm:
$$ \|u\|_\lip = \|u\|_{\infty} + \sup_{x\neq y} \frac{|u(x)-u(y)|}{d(x,y)}
 = \|u\|_{\infty} + |u|_\ell.
$$

We consider a map $T:X\to X$. We suppose that there exists a family of disjoint
open intervals $(a_j,b_j)$, $j\in J$ where $J$ is finite or countable, such that
$X=S\cup\bigcup_j (a_j,b_j)$, where $S$ is at most countable (or of null Lebesgue-measure). On each
$(a_j,b_j)$, the map $T$ is supposed to be differentiable and $T'x\neq 0$ for all $x\in(a_j,b_j)$.
We define $g(x)=-\log |T'x|$ for $x\in(a_j,b_j)$. The value of $g$ on $\{a_j,b_j\}$ is immaterial.
We also suppose that $T$ is
strictly monotonic on $(a_j,b_j)$ and maps the open interval $(a_j,b_j)$ onto $(0,1)$. We will denote
$\sigma_j:[0,1]\to [a_j,b_j]$ the inverse map of $T$ on $(a_j,b_j)$.
The change of variables formula implies that for all 
lipschitz functions $u$, $v$ on $X$ (or more generally, for $v\in L^\infty$ and $u\in L^1$)
\begin{equation} \label{eq-transfer-dual}
\int_X v\circ T (x). u(x) dx = \int_X v(x). \mathcal{L}u(x) dx.
\end{equation}
In (\ref{eq-transfer-dual}), $\LL$ is the
associated transfer operator and is defined by 
\begin{equation} \label{eq-transfer}
\LL u (x) = \sum_{j\in J} e^{g(\sigma_j x)} u(\sigma_j x).
\end{equation}

We make the following assumption.
\begin{assumption} \label{Assumption}\indent
\begin{enumerate}
\def\theenumi{(A\arabic{enumi})}
\item\label{A1} There exists $\gamma>1$ such that
for all $x,y\in X$, all $j\in J$,  $d(\sigma_j x,\sigma_j y) \leq \gamma^{-1} d(x,y)$;
\item\label{A2} There exists $G<\infty$, $G>0$ such that $\sup_{j\in J} |g\circ \sigma_j|_\ell\leq G$;
\item\label{A3} $\sup_{x\in X} \sum_j \exp(g(\sigma_j x))<\infty$.
\end{enumerate}
\end{assumption}
If $g\in\lipX$ then one may take $G=|g|_\ell \gamma^{-1}$ but when $J$ is countable, this is often too strong
a requirement, see example below.
 Since the diameter of $X$ is bounded, by \ref{A2}, Condition \ref{A3} holds as soon as
$\sum_j \exp(g(\sigma_j x))<\infty$ for some $x\in X$. For the potential $g=-\log |T'|$, this is automatic by
(\ref{eq-transfer-dual}).

\begin{remark}
Our proof is written for expanding maps on the interval but this is only a matter of presentation.
In the above, one can instead consider that
$(X,d)$ is any compact metric space not reduced to a single point, and whose diameter is not greater than
$1$. We assume then that $T:X\to X$ is continuous. We consider any map
$g:X\to\R$ and any family of maps $(\sigma_j)_{j\in J}$ ($J$ finite or
countable), $\sigma_j:X\to X$, such that $T\sigma_j (x)=x$ for all $x$,
 and satisfying the assumptions \ref{Assumption}.
The transfer operator $\LL$ is then defined for bounded functions $u$
by (\ref{eq-transfer}). Equation (\ref{eq-transfer-dual}) does
not make any sense in this setting and is replaced by the following, which is valid for all continuous
$u$, $v:X\to \C$,
\begin{equation*}
\LL[u\circ T\cdot v] = u \LL[v].
\end{equation*}
\end{remark}

\begin{example}
Consider the Gauss map $Tx=\{1/x\} = 1/x-\lfloor 1/x\rfloor$, and $T0=0$. On the interval $(1/(j+1),1/j)$,
we have $Tx= 1/x -j$ and $T'x=-1/x^2$. Hence, $g(x)=2\log x$, and for all $j\geq 1$, $\sigma_j(x)=1/(j+x)$.
Observe that, since $g$ is unbounded,
$|g|_\ell=\infty$ for all bounded metric $d$ on $X$ compatible with the topology of $X$. The maps $\sigma_1$ is
$1$-lipschitz for the usual metric on $X$.
However, using Mather's trick (see \cite{Math68}, and also
\cite{Rugh96}), one can construct an equivalent metric $d$ on $X$ for which \ref{A1} is satisfied.
Alternatively, one can use the following metric.
Let $\alpha\in(0,1/2)$ and consider the metric $(1-\alpha-\alpha s)ds$ or equivalently
$$ d_\alpha(x,y) = |x-y|(1-\alpha-\alpha(x+y)/2). $$
Then \ref{A1} is satisfied for $d_\alpha$ with $\gamma^{-1} = 1-\frac{5\alpha}{4}$.
 \ref{A2} holds since
$g\circ\sigma_j(x)=-2\log(j+x)$, and one may take $G=2(1-2\alpha)^{-1}$.
\end{example}

Using \ref{A1}-\ref{A3}, we get that
$\LL\in L(\lipXC)$ (where $L(\lipXC)$ denotes the set of all bounded linear operators
$\lipXC\to\lipXC$) and we
have $\|\LL\|_\lip \leq (1+Ge^G)\|\LL 1\|_\infty$.
The norm of $\LL$ when acting on $C(X;\C)$ (the
 Banach algebra of complex valued continuous functions on $X$ endowed with
$\|\cdot\|_\infty$) is given by $\|\LL 1\|_\infty$.

Let $f\in\lipX$ be a fixed observable.
We define the perturbated transfer operator
$\mathcal{L}(z):\lipXC\to\lipXC$, $z\in\C$,
by $$ \big[\mathcal{L}(z) u\big] (x) = \sum_{j\in J}e^{g(\sigma_j x) +
zf(\sigma_j x)} u(\sigma_j x) = \big[\LL(0)e^{zf}u\big](x). $$

When acting on the Banach algebra $\lipX$ of real-valued lipschitz functions on $X$,
the transfer operator has a spectral gap
(see \cite{Bo75}, \cite{Rue78}, \cite{Liv95} or
\cite{Zin99}).
(This is for instance because $\LL$ is a strict contraction
for the Hilbert metric of the cone $\CC_\R$ which is of
bounded aperture and of non-empty interior, see section \ref{section-estdiam}).
 More precisely, there exist $\lambda_0>0$, $h_0\in\lipX$, $h_0>0$, and $\nu_0\in\lipX'$
such that $ \mathcal{L} h_0 = \lambda_0 h_0 $,
$ \nu_0\mathcal{L} = \lambda_0 \nu_0 $,
and $ \langle \nu_0, h_0 \rangle = 1$. The remaining spectrum is contained in a disk of radius strictly smaller 
than $\lambda_0$. The operator $\lambda_0^{-n} \LL^n$ converges to the one-dimensional projection $h_0\otimes
\nu_0$ with exponential speed of convergence.
Moreover, the functional $m_0\in\lipX'$ defined by
$$\langle m_0, f \rangle = \langle \nu_0,  f h_0 \rangle $$
is a nonnegative linear functional on $\lipX$ and extends to a probability measure on $X$ which
is called the Gibbs state associated to the potential $g$.
For expanding maps on the interval, we normalize $h_0$ so that
$\int h_0(x)dx=1$. Then (\ref{eq-transfer-dual}) implies that $\lambda_0=1$ and that $\nu_0$ is
the Lebesgue measure on $X$.
In general, the measure $m_0$ is  a mixing (hence ergodic) $T$-invariant measure on $X$, see \cite{Bo75}.
We will denote $\esp [u]$ the expectation of $u$ with
respect to this invariant measure $m_0$.

Assume that the fixed observable $f\in\lipX$ satisfies $\esp[f] = 0$. Denote
$$ S_n f = \sum_{k=0}^{n-1} f\circ T^k. $$
In this situation, the following limit exists
\begin{equation} \label{eq-sigma}
 \sigma^2 = \lim_{n\to\infty} \frac{1}{n} \esp[(S_n f)^2].
\end{equation}
Moreover, $\sigma^2=0$  if and only if $f$ is a cocycle: $f=u\circ T- u$, $u\in L^2$. This last statement
along with the existence of the limit in (\ref{eq-sigma}) are, for instance, consequences of
Gordin's approximation by martingales (see \cite{LI96}, \cite{Go86}, \cite{IL71},
see also \cite{Br96} for a different proof).
More precisely, the exponential decay of correlations given by the spectral gap
property of the transfer operator shows that
one can write $$ f = \xi + u\circ T - u,$$
where $u\in L^2$ and $\esp[\xi\circ T^j | \xi\circ T^{j+1}, \xi\circ T^{j+2} \dots]=0$ (in other words
$\xi\circ T^j$ is a reversed martingale difference). The limit (\ref{eq-sigma}) is then
precisely $\sigma^2=\esp[\xi^2]$. 

\section{Complex cones} \label{section-ccone}

We recall in this section some definitions and material regarding
complex cones which can be found in \cite{Rugh07} and
\cite{Dub08}. We assume however that the reader is familiar with the Hilbert metric
(see \cite{Bir57}, \cite{Bir67} or \cite{Liv95}). A non-empty subset $\CC$ of a complex Banach space $V$ is said to be a {\em complex cone}
if $\C^*\CC\subset \CC$. We will also assume here that $0\notin \CC$. The cone $\CC$ is said to be
{\em proper} if its closure $\overline{\CC}$ does not contain any complex subspaces of dimension $2$.
The {\em dual complement} $\CC'\subset V'$ of $\CC$ is the set of all continuous linear functionals not
vanishing on $\CC$. The cone $\CC$ is said to be {\em linearly convex} if for any $x\notin \CC$, one can
find $f\in\CC'$ vanishing at $x$. 
In other words, for a complex cone, one has
$$ f\in\CC' \quad\iff\quad \forall x\in\CC,\:\:\langle f,x\rangle \neq 0, $$
and for a linearly convex cone one also has
$$ x\in\CC\quad\iff\quad \forall f\in\CC',\:\:\langle f,x\rangle \neq 0.$$
The dual complement, when non-empty, is always linearly convex.

Recall also the definition of the projective metric $\delta_{\mathcal{C}}$ of a proper complex
cone $\CC$. 
Let $x$, $y\in\mathcal{C}$, then $\delta_{\mathcal{C}}(x,y) = \log (b/a)\in [0,\infty]$,
 where $b$ and $a$ are respectively the
largest and smallest modulus of the set
$$E(x,y) = E_{\mathcal{C}}(x,y)=\{z\in\C:\: zx-y \notin \mathcal{C}\}.$$
When the cone is linearly convex, $\delta_\CC$ satisfies the triangular inequality and thus is really a
projective metric\footnote{
$\delta_{\mathcal{C}}$ is called a metric even though it may take infinite values.
Projective means that for any scalar $\alpha$, $\delta_\CC(\alpha x,y)=\delta_\CC(x,\alpha y)
=\delta_\CC(x,y)$.}.
When the cone $\CC$ is linearly convex, one also has the following description of $E_\CC(x,y)$
\begin{equation} \label{eq-ECxy}
 E_\CC(x,y) = \left\{ \frac{\langle f,y\rangle}{\langle f,x\rangle}:\:f\in \CC'\right\}.
\end{equation}
See \cite{Dub08} for more details. 

Finally, recall (see \cite{Rugh07}) that the cone $\mathcal{C}$ is said to be of
$K$-bounded sectional aperture if for
each vector subspace $\PP$ of (complex) dimension 2, one may find $m=m_\PP \in V'$, $m\neq 0$ such that
\begin{equation} \label{eq-KBSA}
\forall u\in \CC\cap \PP,\quad 
\|m\|\cdot\|u\| \leq K |\langle m,u\rangle|.
\end{equation}
When $m$ can be chosen independent of $\PP$,
or equivalently when $B(m,K^{-1}\|m\|)\subset\mathcal{C}'$,
$\CC$ is said to be of $K$-bounded (global) aperture.
The following Theorem is proved in \cite{Dub08}
(inequality \ref{p1} is established in the proof of Lemma 2.2
of \cite{Dub08}).

\begin{theorem} \label{th-contraction-delta}
\begin{enumerate}
\item \label{p2} Suppose that the cone $\CC$ is linearly convex and of bounded sectional aperture. Then
$(\CC/\sim,\delta_\CC)$ is a complete metric space, where $x\sim y$ if and only if $\C^* x= \C^* y$.
\item \label{p1} Suppose that the cone $\mathcal{C}$ is of 
$K$-bounded aperture and let
$m\in V'\MZ$ such that (\ref{eq-KBSA}) holds. Then for all 
$x$, $y\in\mathcal{C}$,
$$ \left\| \frac{x}{\langle m,x \rangle} - \frac{y}{\langle m,y\rangle} \right\| \leq \frac{K}{2\|l\|}
\delta_{\mathcal{C}}(x,y). $$
\item \label{p0} Let $A:V_1\to V_2$ be a complex linear map such that
$A \mathcal{C}_1 \subset \mathcal{C}_2$.
 Then for all $x$, $y\in\mathcal{C}_1$, $$\delta_{\mathcal{C}_2}(Ax, Ay) \leq \tanh(\Delta/4)
\delta_{\mathcal{C}_1}(x,y),$$ where $\Delta = \sup_{x,y\in\mathcal{C}_1} \delta_{\mathcal{C}_2}(Ax,Ay)$.
\end{enumerate}
\end{theorem}

\section{Comparison of operators} \label{sect-comp}

Let $V_{\R}$ be a real Banach space and $V_{\C}=V_{\R} \oplus iV_{\R}$ its complexification. We consider a
real nontrivial (meaning that it contains at least two independent vectors) closed proper convex cone 
$\mathcal{C}_{\R}\subset V_{\R}$ with dual cone $\mathcal{C}'_{\R}$. Recall
(see \cite{Rugh07}) the definition of the canonical complexification $\mathcal{C}_\C\subset V_{\C}$ of
$\mathcal{C}_{\R}$:
$$ \mathcal{C}_\C = \{ x\in V_{\C}:\: \forall l_1,\,l_2\in \mathcal{C}'_{\R},\, \Re\left(\langle
l_1,x\rangle \overline{\langle l_2,x \rangle} \right) \geq 0 \}. $$
The canonical complexification also satisfies $\mathcal{C}_\C
 = \C^*( \mathcal{C}_{\R} + i \mathcal{C}_{\R})$.
The complex cone $\CC_\C\MZ$ is a proper complex cone. 
In what follows, we consider $P:V_\R\to V_\R$ a real linear operator mapping $\CC_\R\MZ$ into itself, and
$A:V_\C\to V_\C$ a complex linear operator. We denote $\CC=\CC_\C\MZ$.

\begin{lemma} \label{prop-linearlyconvex}
Assume that there exists a linear functional $m\in\mathcal{C}'_{\R}$ such that $m>0$ on 
$\mathcal{C}_{\R}\MZ$. Then the cone $\CC=\mathcal{C}_\C\MZ \subset V_{\C}$ is linearly convex.
\end{lemma}

We prove Lemma \ref{prop-linearlyconvex} in Appendix \ref{app-linconv}. When the condition of the Lemma fails, $\delta_\CC$ still satisfies the
triangular inequality on $\CC_\C$, see Remark 4.7 of \cite{Dub08}.

\begin{lemma} \label{lem-dh}
Let $x$, $y\in\CC_\R\MZ$ be independent. Then
\begin{enumerate}
\item \label{pdiam1}
$E_\CC(x,y)$ is the open disk of diameter $(a,b)$ where $0\leq a\leq b\leq \infty$ (if $b=\infty$,
it is the half-plane $\{w:\Re(w)> a\}$).
\item \label{pdiam2} $\delta_\CC(x+iy, x) \leq \delta_\CC(x,y)$.
\end{enumerate}
\end{lemma}

\begin{proof}
\ref{pdiam1}.\ This is established in the proof of Theorem 5.5 of \cite{Rugh07}
(with different coordinates for the sections). It
is also a consequence of Lemma 4.1 of \cite{Dub08}.
The claimed result holds with $b=\sup\frac{\langle m,y\rangle}{\langle m,x\rangle }\leq \infty$ and
and $a=\inf\frac{\langle m,y\rangle}{\langle m,x\rangle }\geq 0$ where both the supremum and the infimum are taken
over all $m\in\CC_\R'$ such that $\langle m,x\rangle >0$.

\ref{pdiam2}.\ One may assume $\delta_\CC(x,y)<\infty$.
Then $E_\CC(x,y)$ is an open disk whose diameter is some interval $(a,b)$ such that $0<a< b<\infty$.
We have $E_\CC(x,x+iy) = 1 + iE_\CC(x,y)$. So $E_\CC(x,x+iy)$ is the open disk of center
$i(a+b)/2+1$ and radius $(b-a)/2$. Therefore, we have
$$ \delta_\CC(x,x+iy) = \log\frac{|1+i\frac{b+a}{2}| +\frac{b-a}{2}}
{|1+i\frac{b+a}{2}| -\frac{b-a}{2}}
\leq \log \frac{\frac{b+a}{2} + \frac{b-a}{2}}{\frac{b+a}{2}-\frac{b-a}{2}} = \delta_\CC(x,y). $$
\end{proof}

\begin{theorem} \label{thm-dh}
\begin{enumerate}
\item $ (\CC_\R\MZ, h_{\CC_\R}) \hookrightarrow (\CC_\C\MZ, \delta_{\CC_\C}) $
is an isometric embedding (where $h_{\CC_\R}$ denotes the Hilbert metric of $\CC_\R$).
\item The natural extension of $P$ to $V_\C$ (still denoted by $P$) maps the cone $\CC_\C\MZ$ into itself.
Let $\Delta_\R$ (resp. $\Delta_\C$) be the diameter of $P(\CC_\R\MZ)$ (resp. $P(\CC_\C\MZ)$)
for the Hilbert metric (resp. the projective metric $\delta$). Then
\begin{equation} \label{eq-esti-delta}
 \Delta_\C\leq 3\Delta_\R.
\end{equation}
\end{enumerate}
\end{theorem}

\begin{proof}
Denote $\CC=\CC_\C\MZ$.
Let $x$, $y\in\CC_\R\MZ$ be independent. Then the fact that $\delta_\CC(x,y)=h_{\CC_\R}(x,y)$ is
a consequence of Lemma \ref{lem-dh}\ref{pdiam1}.
Finally, let $w_1$, $w_2\in\CC$. Write $w_j=e^{i\theta_j}(x_j+iy_j)$, $x_j$, $y_j\in\CC_\R$.
Since $\delta_\CC$ satisfies the triangular inequality, and using Lemma \ref{lem-dh}\ref{pdiam2}, we have
(in the case where $x_1\neq 0$, $x_2\neq 0$)
\begin{eqnarray*}
 \delta_\CC(A w_1,A w_2) &\leq &\delta_\CC(A x_1 + iA y_1, A x_1)
+\delta_\CC(A x_1, A x_2) + \delta_\CC(A x_2, A x_2 + iA y_2)\\
&\leq& \delta_\CC(A x_1,A y_1) +\delta_\CC(A x_1,A x_2)
+\delta_\CC(A x_2, A y_2) \leq 3\Delta_\R.
\end{eqnarray*}
\end{proof}

\begin{lemma} \label{prop-comp}
Assume that $A$ does not vanish on the cone $\CC$. Assume also that there exists
$\tau$, $0\leq \tau<1$, such that for all $m$, $l\in\CC_\R'$ and all $x\in\CC$,
\begin{eqnarray} \label{eq-compare}
\big|\langle m,Px\rangle\langle l, Ax \rangle - \langle m,Ax\rangle\langle l, Px \rangle\big|
\leq\tau \Re\big(\overline{\langle m,Px\rangle}\langle l, Ax \rangle
+\overline{\langle l,Px\rangle}\langle m, Ax \rangle\big).
\end{eqnarray}
Then $A\CC\subset \CC$ and we have
$$ \sup_{x\in\CC} \delta_\CC(Ax,Px) \leq 3\log\frac{1+\tau}{1-\tau}. $$
\end{lemma}

\begin{remark}
Condition (\ref{eq-compare}) is stable under convex combinations. Therefore, it is sufficient to check
this condition for $m$, $l$ belonging to some generating subset $\mathcal{S}$ of
$\CC_\R'$. By a generating subset $\CS$ of $\mathcal{C}'_{\R}$, we mean 
that $\mathcal{C}_{\R}=\{x\in V_{\R}:\: \langle l,x\rangle\geq 0,\, \forall l\in\mathcal{S}\}$, or
equivalently, that $\mathcal{C}'_{\R} = \textrm{Cl}_{w*}\left( \R_+ \textrm{ch}(\mathcal{S})\right)$
where `$\textrm{ch}$' means `convex hull' and $\textrm{Cl}_{w*}$ denotes the closure with respect to the
weak-* topology.
\end{remark}

\begin{proof}
Fix $x\in\CC=\CC_\C\MZ$. Write $x=e^{i\theta}(u+iv)$, $u$, $v\in\CC_\R$. Pick up $\mu\in\CC_\R'$ for which
$\langle \mu, Pu\rangle +\langle \mu, Pv\rangle>0$ (recall that since $\CC_\R$ is proper, for any $w\in\CC_\R\MZ$,
one can find $\nu\in\CC_\R'$ such that $\langle \nu,x\rangle >0$).
Let $m$, $l\in\CC_\R'$, and suppose that for some $\epsilon>0$, we have
$m$, $l\geq \epsilon\mu$ on the cone $\CC_\R$. Write
\begin{equation} \label{eq-not}
 \begin{pmatrix}
\langle m,Px\rangle & \langle m,Ax\rangle\\
\langle l, Px \rangle & \langle l, Ax \rangle
\end{pmatrix}
= \begin{pmatrix}
a & b\\ c& d
\end{pmatrix}
\end{equation}
Our assumption implies that $|\overline{a}d +\overline{c}b| \geq |ad-bc|$, or equivalently
$2\Re(\overline{a}dc\overline{b})\geq - 2\Re(\overline{a}\overline{d}bc)$. So we have
$4\Re(\overline{a}c)\Re(\overline{b}d)\geq 0$.  We then get
$$ \Re(\overline{a}c)\geq \epsilon^2(\langle \mu, Pu\rangle^2 + \langle \mu, Pv\rangle^2)
= \epsilon^2 |\langle \mu, Px\rangle|^2 >0. $$
Therefore, $\Re(\overline{\langle m,Ax\rangle}\langle l,Ax\rangle)\geq 0$ for all $m$, $l\geq \epsilon\mu$.
This is also true for arbitrary $m$, $l\in\CC_\R'$ since we can apply the above argument to
$m+\epsilon\mu$, $l+\epsilon\mu$ and let $\epsilon\to 0$. This proves that $Ax\in \CC_\C\MZ$, since by assumption
we have also $Ax\neq 0$.
We turn now to estimate the distance between $Ax$, and $Px$.
Let $\FF$ be the family of all couple $(m,l)$ such that
$|ad-bc|>0$ (using the notations (\ref{eq-not})).
Then, by Lemma 4.1 of \cite{Dub08}, $E_{\CC_\C}(Px,Ax)=\bigcup_{(m,l)\in\FF} D_{m,l}$. Here
$D_{m,l}=\varphi_{m,l}(\{w:\Re(w)>0\})$ and $\varphi_{m,l}$ is the Möbius transformation given by
(using the notations (\ref{eq-not}))
$$ z= \varphi_{m,l}(w) = \frac{wb+d}{wa+c}. $$
Let $(m,l)\in\FF(Px,Ax)$. Our assumption yields $0<|ad-bc|\leq \tau \Re(\overline{a}d +\overline{c}b)$
which forces $\Re(\overline{a}c)\neq 0$. Therefore
the Möbius transformation $\varphi_{m,l}$ maps the half-plane
$\{w:\Re(w)>0\}$ onto the open disk $D_{m,l}$ of center $c_{m,l}$ and radius $r_{m,l}$ given by
$$ c_{m,l} =\frac{\overline{a}d+\overline{c}b}{2\Re(\overline{a}c)},
\quad\textrm{ and }\quad
r_{m,l} = \frac{|ad-bc|}{2\Re(\overline{a}c)}.$$
Then we have
$$ \frac{\sup|D_{m,l}|}{\inf|D_{m,l}|} = \frac{|c_{m,l}| +r_{m,l}}{|c_{m,l}| -r_{m,l}}
\leq \frac{1+\tau}{1-\tau}. $$
If $(m',l')\in\FF$, then
$\langle l,Ax\rangle /\langle l,Px\rangle=\varphi_{m,l}(0)=\varphi_{m',l}(0) $
belongs to both $\overline{D}_{m,l}$ and $\overline{D}_{m',l}$; and
$\langle m',Ax\rangle /\langle m',Px\rangle $ belongs to both $
\overline{D}_{m',l}$ and $\overline{D}_{m',l'}$. We thus have
$$ \frac{\sup|D_{m,l}|}{\inf|D_{m',l'}|} \leq \left(\frac{1+\tau}{1-\tau}\right)^3. $$
And the proposition follows.
\end{proof}

\begin{theorem} \label{prop-comp2}
Assume that the diameter $\Delta_P$ of $P(\CC_\R\MZ)$ with respect to the Hilbert metric of $\CC_\R$ is finite.
Assume also that there exists $\epsilon>0$ such
that we have for all $m\in\CC_\R'$ and
all $u\in\CC_\R$
\begin{equation} \label{eq-compare2}
 \left| \langle m,Au\rangle - \langle m,Pu\rangle \right| \leq\epsilon \langle m,Pu\rangle.
\end{equation}
Then if
$$
2\epsilon (1+\cosh(\Delta_P/2)) <1,
$$ then $A \CC\subset \CC$, and we have
$$ \sup_{x\in\CC} \delta_\CC(Ax,Px) \leq 3\log\frac{1}{1-2\epsilon(1+\cosh(\Delta_P/2))}. $$
\end{theorem}

\begin{remark}
The condition (\ref{eq-compare2}) is also stable under convex combinations, so one only has to check it
for a generating subset of $\CC_\R'$.
\end{remark}

\begin{proof}
We show that the assumptions of Lemma \ref{prop-comp} are satisfied.
Let $x\in\CC_\C\MZ$ and write $x=e^{i\varphi}(u+iv)$, $u$, $v\in\CC_\R$. Up to modifying $\varphi$, one
might assume $u\neq 0$, $v\neq 0$.
First we prove that $Ax\neq 0$. 
Pick $\mu\in\CC_\R'$ for which $\langle \mu,Pu\rangle>0$ and
$\langle \mu,Pv\rangle >0$.
Then, the condition on $\epsilon$ forces $\epsilon<1/4\leq \sin(\pi/12)$.
So (\ref{eq-compare2}) implies $\langle \mu,Au\rangle\neq 0$ and
$|\arg \langle \mu, Au\rangle|\leq \pi/12$.
The same is true for $v$ so we cannot have $\langle \mu, A(u+iv)\rangle =0$.
Now we establish (\ref{eq-compare}).
Let $m$, $l\in\CC_\R'$.
The following inequality is established in the proof of Theorem 6.3 in \cite{Rugh07}
\begin{equation} \label{eq-pol}
|\langle m,Px\rangle\langle l,Px\rangle| \leq \cosh\left(\frac{\Delta_P}{2}\right)
\Re(\overline{\langle m,Px\rangle}\langle l,Px\rangle).
\end{equation}
To establish (\ref{eq-compare}), one might assume that $m$, $l>0$ on
$\{Pu,Pv\}$ (otherwise, consider $m+t\mu$, $l+t\mu$, $t>0$, $t\to0$).
Now we have
\begin{eqnarray*}
&&\Re\big(\overline{\langle m,Px\rangle}\langle l, Ax \rangle
+\overline{\langle l,Px\rangle}\langle m, Ax \rangle\big) \\
&=&
\langle m,Pu\rangle\langle l,Pu\rangle\Re\left(
\frac{\langle l,Au\rangle}{\langle l,Pu\rangle}+
\frac{\langle m,Au\rangle}{\langle m,Pu\rangle}
\right)+
\langle m,Pv\rangle\langle l,Pv\rangle\Re\left(
\frac{\langle l,Av\rangle}{\langle l,Pv\rangle}+
\frac{\langle m,Av\rangle}{\langle m,Pv\rangle}
\right)\\
&+&
\langle m,Pv\rangle\langle l,Pu\rangle\Im\left(
\frac{\langle l,Au\rangle}{\langle l,Pu\rangle}-
\frac{\langle m,Av\rangle}{\langle m,Pv\rangle}
\right)-
\langle m,Pu\rangle\langle l,Pv\rangle\Im\left(
\frac{\langle l,Av\rangle}{\langle l,Pv\rangle}-
\frac{\langle m,Au\rangle}{\langle m,Pu\rangle}
\right)\\
&\geq& 2(1-\epsilon)\Big(\langle m,Pu\rangle\langle l,Pu\rangle
+\langle m,Pv\rangle\langle l,Pv\rangle\Big)
-2\epsilon\Big(\langle m,Pv\rangle\langle l,Pu\rangle
+\langle m,Pu\rangle\langle l,Pv\rangle\Big)\\
&=&2(1-\epsilon)\Re\Big(\overline{\langle m,Px\rangle}{\langle l,Px\rangle}\Big)
-2\epsilon \Im\Big(\langle m,Px\rangle\langle l,Px\rangle\Big)\\
&\geq& 2\Big(1-\epsilon\left(1+\cosh\left(\Delta_P/2\right)\right)\Big)
\Re\Big(\overline{\langle m,Px\rangle}{\langle l,Px\rangle}\Big).
\end{eqnarray*}
In the same way, we have
\begin{eqnarray*}
&&\big|\langle m,Px\rangle\langle l, Ax \rangle
-\langle l,Px\rangle\langle m, Ax \rangle\big| \\
&=&\Big|\langle m,Pu\rangle\langle l,Pu\rangle\left(
\frac{\langle l,Au\rangle}{\langle l,Pu\rangle}-
\frac{\langle m,Au\rangle}{\langle m,Pu\rangle}
\right)+
\langle m,Pv\rangle\langle l,Pv\rangle\left(
\frac{\langle m,Av\rangle}{\langle m,Pv\rangle}-
\frac{\langle l,Av\rangle}{\langle l,Pv\rangle}
\right)\\
&+&
i\langle m,Pv\rangle\langle l,Pu\rangle\left(
\frac{\langle l,Au\rangle}{\langle l,Pu\rangle}-
\frac{\langle m,Av\rangle}{\langle m,Pv\rangle}
\right)+
i\langle m,Pu\rangle\langle l,Pv\rangle\left(
\frac{\langle l,Av\rangle}{\langle l,Pv\rangle}-
\frac{\langle m,Au\rangle}{\langle m,Pu\rangle}
\right)\Big|\\
&\leq&2\epsilon\Re\Big(\overline{\langle m,Px\rangle}{\langle l,Px\rangle}\Big)
+2\epsilon \Big|\langle m,Px\rangle\langle l,Px\rangle\Big|\\
&\leq& 2\epsilon(1+\cosh(\Delta_P/2))\Re\Big(\overline{\langle m,Px\rangle}{\langle l,Px\rangle}\Big).
\end{eqnarray*}
So we can apply Lemma \ref{prop-comp} with
$$ \tau = \frac{\epsilon(1+\cosh(\Delta_P/2))}{1-\epsilon(1+\cosh(\Delta_P/2))}. $$
\end{proof}

\begin{remark} \label{remark-estimates}
 Denote by $d_\CC$ Rugh's hyperbolic gauge (see \cite{Rugh07} for definitions).
Then $(\CC_\R\MZ, h_{\CC_\R}) \hookrightarrow (\CC_\C\MZ, d_{\CC_\C})$ is also an isometric embedding.
Regarding the second part of Theorem \ref{thm-dh},
if $\Delta_\R<\infty$, then so is the diameter $\Delta_{\textrm{hyp}}$ of $A(\CC_\C\MZ)$ for $d_\CC$
(see \cite{Rugh07}, Proposition 5.9). It is also possible to give a general bound 
for $\Delta_{\textrm{hyp}}$ : as a consequence of Proposition 5.7 of \cite{Dub08} and Theorem \ref{thm-dh},
$$ \Delta_{\textrm{hyp}} \leq \pi\sqrt{2} \exp(3\Delta_\R/2). $$
The constants in the preceding inequality might not be optimal, but one cannot get rid
of the exponential. Indeed, in Remark 5.8 of \cite{Dub08}, we provided an example of a sequence of positive $3\times 3$ matrix $A_k$ for which
the Hilbert diameter (and also the $\delta_\CC$-diameter)
is $O(\log k)$ but the $d_\CC$-diameter is not lesser than $k\log 2$. So in general, (\ref{eq-esti-delta})
is much better than what can be obtained with hyperbolic gauges. 

\end{remark}

\section{Cones and estimate of diameters} \label{section-estdiam}

Following \cite{Rugh07}, we define also the real cone $\mathcal{C}_{\R} \subset \lipX$ by
$$ \mathcal{C}_{\R} = \{ u\in \lipX :\: \langle l_{x,y}, u\rangle \geq 0,\, \forall x,y\in X\}, $$
where $\langle l_{x,y} , u \rangle = e^{Bd(x,y)} u(y) - u(x)$.
We define $B$ to be
$$ B = \frac{\gamma G+1}{\gamma -1}>0. $$
Actually, $B>\gamma G/(\gamma-1)$ would be sufficient, we make this particular choice to have a simple expression
for $\epsilon(z)$ below.
Observe that if $u\in \mathcal{C}_{\R}$ then $u(x)\geq 0$, $\forall x$.
We denote by $\mathcal{C}_\C\subset \lipXC$ 
the canonical complexification of the cone $\mathcal{C}_{\R}$, and we define
$$\mathcal{C} = \CC_\C\MZ = \C^*(\mathcal{C}_\R + i \mathcal{C}_\R)\MZ. $$
Recall also that $\CC_\C=\{u\in\lipXC:\:\Re\big(\overline{\langle l_{x,y},u\rangle}\langle l_{x',y'},u\rangle
\big)\geq 0,\forall x,y,x',y'\in X\}$.

The cone $\CC$ is of bounded aperture and linearly convex and the cone $\CC'$ is also of bounded
aperture. Indeed,
let fix $y\in X$. Let $u\in\mathcal{C}_{\R}$. Then $\|u\|_{\infty} \leq e^B u(y)$. Moreover
$$ u(x)-u(x')\leq u(x)(1-e^{-Bd(x,x')}) \leq u(x)Bd(x,x') \leq Be^Bu(y)d(x,x').$$
Therefore, $|u|_{\textrm{Lip}}\leq Be^Bu(y)$, and $\|u\|_\lip\leq (B+1)e^B u(y)$. Let now $w\in\mathcal{C}$, 
then $w=e^{i\alpha}(u+iv)$, $u$, $v\in\mathcal{C}_{\R}$. So
\begin{equation} \label{eq-regext}
\|w\|_\lip \leq (B+1)e^B(u(y)+v(y)) \leq \sqrt{2}(B+1)e^B |w(y)|=K|w(y)|.
\end{equation}
So the cone $\mathcal{C}$ is linearly convex by Proposition \ref{prop-linearlyconvex} since the 
linear functional $u\mapsto u(y)$ is positive on $\mathcal{C}_{\R}\MZ$ for any $y\in X$.
Denoting $C_1=\max(1,Be^B)>0$, we have for all $h\in\lipXC$ and $x,y\in X$,
\begin{equation} \label{eq-estim-int-2}
 |\langle l_{x,y},h\rangle|\leq |h(y)|(e^{Bd(x,y)}-1) + |h(y)-h(x)|\leq C_1\|h\|_\lip d(x,y).
\end{equation}
Therefore, $\Re\big(\overline{\langle l_{x,y},1+h\rangle}\langle l_{x',y'},1+h\rangle\big)
\geq \big(B^2-2C_1^2\|h\|-C_1^2\|h\|^2\big) d(x,y)d(x',y')$. Hence,
the constant function $1$ is in the interior of $\mathcal{C}$,
say $B(1,1/K')\subset\mathcal{C}$. Thus $\mathcal{C}'$ 
is of $K'$-bounded aperture: for all $m\in\mathcal{C}'$,
\begin{equation} \label{eq-regint}
\|m\|\leq K'|\langle m,1\rangle|.
\end{equation}

\begin{lemma} \label{lemma-epsilonz}
Let $x$, $y\in X$, $x\neq y$, $u\in\CC_\R\MZ$, and $z\in\C$. Then
\begin{equation*}
\left| \frac{\langle l_{x,y}, \LL(z)u\rangle}{\langle l_{x,y}, \LL u\rangle} -1\right| \leq \epsilon(z),
\end{equation*}
where $\epsilon(z)$ is given by
\begin{equation*}
\epsilon(z) = e^{|\Re(z)|\|f\|_\infty} |z| \left(\|f\|_\infty + 
|f|_\ell\right).
\end{equation*}
\end{lemma}

\begin{proof}
We have
\begin{eqnarray}
\frac{\langle l_{x,y}, \LL(z)u\rangle}{\langle l_{x,y}, \LL u\rangle}
&=& \frac{\sum_j \langle l_{j},u\rangle Z_{j}}
{\sum_j \langle l_{j},u\rangle}, \label{eq-cczj}
\end{eqnarray}
where
\begin{eqnarray}
\langle l_j,u\rangle &=& e^{Bd(x,y)+g(\sigma_j y)}u(\sigma_j y) - e^{g(\sigma_j x)}u(\sigma_j x)>0,
\nonumber\\
Z_j&=&\frac{e^{t_j}e^{zf(\sigma_j y)} - e^{zf(\sigma_j x)}}
{e^{t_j}-1},\nonumber\\
t_j&=&B d(x,y) + g(\sigma_j y) - g(\sigma_j x) +\log u(\sigma_j y) - \log u(\sigma_j x)
\geq 
\frac{d(x,y)}{\gamma}.
\label{eq-tj}
\end{eqnarray}
Now, we have
\begin{eqnarray}
|Z_j - 1| &=& \left|e^{z f(\sigma_j y)} -1 + \frac{e^{z f(\sigma_j y)} - e^{z f(\sigma_j x)}}
{e^{t_j}-1} \right| \nonumber\\
&\leq & e^{|\Re(z)|\|f\|_\infty} |z|\|f\|_\infty + 
\left|\frac{e^{z f(\sigma_j y)} - e^{z f(\sigma_j x)}}{e^{t_j}-1}\right|.\label{eq-zj}
\end{eqnarray}
Suppose for instance that $\Re(zf(\sigma_j y))\geq \Re(zf(\sigma_j x))$, and write
$z\big(f(\sigma_j y)-f(\sigma_j x)\big)=\alpha_j+i\beta_j$, $\alpha_j\geq 0$, $\beta_j\in\R$.
 Then, using (\ref{eq-tj}),
\begin{eqnarray*}
\left|\frac{e^{z f(\sigma_j y)} - e^{z f(\sigma_j x)}}{e^{t_j}-1}\right|
&=& e^{\Re(z)f(\sigma_j y)}\left|\frac{1 - e^{-\alpha_j-i\beta_j} }{e^{t_j}-1}\right|\\
&\leq& e^{|\Re(z)|\|f\|_\infty}\frac{|\alpha_j+i\beta_j|}{t_j} \\
&=&e^{|\Re(z)|\|f\|_\infty}\frac{|z||f(\sigma_j x)-f(\sigma_j y)|}{t_j} 
\leq e^{|\Re(z)|\|f\|_\infty} 
|z||f|_\ell.
\end{eqnarray*}
Combining with (\ref{eq-zj}), we get that $|Z_j-1|\leq \epsilon(z)$. The result follows since the ratio in
(\ref{eq-cczj}) is a (possibly infinite) convex combination of $Z_j$.
\end{proof}

In order to apply Proposition \ref{prop-comp2}, we need the well-known estimate of the Hilbert diameter
of $\LL(\CC_\R\MZ)$.

\begin{lemma} \label{lemma-51}
\begin{enumerate}
\item \label{qdiam1}
$\mathcal{L} (\mathcal{C}_{\R}\MZ) \subset
\mathcal{C}_{\R}\MZ$, and the diameter of $\mathcal{L} (\mathcal{C}_{\R}\MZ)$ with respect to the
Hilbert metric of $\mathcal{C}_{\R}$ is not greater than $D_\R$ ($D_\R$ defined by (\ref{eq-DR})).
\item \label{qdiam2}
The diameter of $\LL\CC$ with respect to $\delta_\CC$ is not greater than $D=3D_\R$. If $u\in\CC_\R\MZ$ and
$w\in \CC$, one has also the better estimate
$\delta_\CC(\LL u, \LL w)\leq 2 D_\R$.
\end{enumerate}
\end{lemma}

\begin{proof} 
\ref{qdiam1}.\ 
This part of the proof is classical. One has for $u\in\mathcal{C}_{\R}\MZ$
\begin{eqnarray}
\big[\mathcal{L}u\big](x) &=& \sum_{j\in J}
e^{g(\sigma_j x)} u(\sigma_j x) 
\leq e^{(G+B\gamma^{-1})d(x,y)}\sum_{j\in J}
 e^{g(\sigma_j y)}
u(\sigma_j y)\nonumber\\
& =& e^{(G+B\gamma^{-1})d(x,y)}\big[\mathcal{L}u\big](y), \label{eq-pos}
\end{eqnarray}
So we have for $x\neq y$, $u$, $v\in\mathcal{C}_{\R}\MZ$
\begin{eqnarray}
\langle l_{x,y}, \mathcal{L} v \rangle &\leq& \mathcal{L}v(y)
\left( e^{Bd(x,y)} - e^{-(G+B\gamma^{-1})d(x,y)} \right), \nonumber\\
\langle l_{x,y}, \mathcal{L} u \rangle &\geq& \mathcal{L}u(y)
\left( e^{Bd(x,y)} - e^{(G+B\gamma^{-1})d(x,y)} \right) > 0. \label{eq-lxy-pos}
\end{eqnarray}
Therefore, we have the following estimate for the diameter with respect to the Hilbert metric.
\begin{eqnarray*}
\textrm{diam }\LL\CC_\R &=& \sup_{u,\,v\in\mathcal{C}_{\R}\MZ} \sup_{x\neq y, x' \neq y'} 
\log \frac{\langle l_{x,y}, \mathcal{L}v\rangle \langle
l_{x',y'}, \mathcal{L}u\rangle}
{\langle l_{x,y}, \mathcal{L}u\rangle \langle l_{x',y'}, \mathcal{L}v\rangle}, \\
&\leq& 2\log\frac{ B + (G+B\gamma^{-1}) }{ B - (G+B\gamma^{-1})} 
+ \sup_{u,v} \sup_{y,y'} \log \frac{ \mathcal{L} v(y) }{\mathcal{L} u(y)}
\frac{\mathcal{L} u(y')}{\mathcal{L} v(y')}, \\
&\leq& 2\log\frac{ B + (G+B\gamma^{-1}) }{ B - (G+B\gamma^{-1})} 
+ 2(G+B\gamma^{-1}) = D_\R <\infty.
\end{eqnarray*}
\medskip

\ref{qdiam2}.\ The first part is the content of Theorem \ref{thm-dh}.
If $u\in\CC_\R\MZ$ and $w\in\CC$, we prove $\delta_\CC(\LL u,\LL w)\leq 2 D_\R$
using Lemma \ref{lem-dh}, \ref{pdiam2} and in the same way as (\ref{eq-esti-delta}).
\end{proof}

\section{Proof of Theorem \ref{main-thm}} \label{section-fourier}


Let $D_\R$ be as in (\ref{eq-DR}) and define $\delta_0>0$ by
\begin{equation} \label{eq-delta0}
 \delta_0(\|f\|_\infty + |f|_\ell) = \frac{1}{3(1+\cosh(D_\R/2))} = \frac{1}{6\cosh^2(D_\R/4)}.
\end{equation}

\begin{lemma} \label{lemma-diameter}
There exists $\Delta_0$, $\Delta_0 \leq 4.65$, such that for all $z\in \C$, $|z|\leq \delta_0$,
\begin{enumerate}
\item the perturbated transfer operator $\LL(z)$ maps the cone $\CC$ into itself. 
\item  $ \sup_{u\in \CC} \delta_\CC(\LL(z)u, \LL(0)u) \leq \Delta_0. $
\end{enumerate}
\end{lemma}

\begin{proof}
Let $|z|\leq \delta_0$ where $\delta_0$ is given by
(\ref{eq-delta0}). Then, the choice of $\delta_0$ implies that $\|f\|_\infty \delta_0 \leq 1/6$. So we have
(where $\epsilon(z)$ is from Lemma \ref{lemma-epsilonz})
$$ \epsilon(z) \leq \frac{e^{1/6}}{3(1+\cosh(D_\R/2))} =: \epsilon_0. $$
Using Lemma \ref{lemma-51}, we apply Proposition \ref{prop-comp2} with $\epsilon_0$ and we may take
$$\Delta_0 = 3\log\frac{1}{1-(2/3)e^{1/6}} \leq 4.65. $$
\end{proof}

\begin{corollary} \label{lemma-eigen}
Let $|z|\leq \delta_0$. Then there exists $h(z)\in\CC$, $\nu(z)\in\CC'$ and $\lambda(z)\in\C^*$ such that
$$\LL(z) h(z) = \lambda(z) h(z),\quad \nu(z)\LL(z) = \lambda(z) \nu(z), $$
and where $\nu(z)$ and $h(z)$ are normalized by
$$ \langle \nu(z),h(z)\rangle =1,\quad \langle \nu(z), 1 \rangle = 1. $$
Moreover, $h(z)$, $\nu(z)$ and $\lambda(z)$ are holomorphic functions of $z$ in the open disc $|z|<\delta_0$.
The eigenvalue $\lambda(z)$ is a simple eigenvalue of the operator $\LL(z)$ and the rest of the spectrum of
$\LL(z)$ is included in a disc of radius strictly smaller than $|\lambda(z)|$.
\end{corollary}

\begin{proof}
From Lemma \ref{lemma-diameter},
the projective diameter of $\LL(z)\CC$ in $\CC$ is finite and uniformly bounded
by $3D_\R+2\Delta_0$ for $|z|\leq \delta_0$. The cone $\CC$ is of bounded aperture and has a non-empty interior.
So $\LL(z)$ has a spectral gap by Theorems 3.6 and 3.7 of \cite{Rugh07} (using the metric $\delta_\CC$ instead
of the hyperbolic gauge). This proves the existence of $h(z)\in\CC$ and
$\nu(z)$. However, we need to know in addition that the left
eigenvector $\nu(z)$ belongs to $\CC'$.
To this end, we notice that as a consequence of Lemma 2.4 of \cite{Dub08},
the projective diameter of $\LL(z)\CC$ in $\CC$
equals the projective diameter of $\LL(z)'\CC'$ in $\CC'$ (where $\LL(z)'$ is the adjoint map).
Therefore, $\LL(z)'$ also is a strict contraction
of the cone $\CC'$ which is linearly convex and of bounded aperture. Thus,
Theorem \ref{th-contraction-delta} shows that $\LL(z)'$ has a unique invariant line in $\CC'$. So we can find a right eigenvector
$\tilde{\nu}(z)\in\CC'$
which must satisfies $\langle \tilde{\nu}(z),h(z)\rangle \neq 0$. Hence, $\tilde{\nu}(z)$ and $h(z)$
are eigenvectors for the same eigenvalue and we must have $\nu(z)=\tilde{\nu(z)}$ (up to a constant).
Finally, the analyticity statement
is a consequence of standard perturbation theory (see \cite{Kato80}).
It can also be proved directly since for instance $\nu(z)$ can be expressed as a uniform limit for $|z|<\delta_0$ of holomorphic map.
Indeed, as a consequence of Theorem \ref{th-contraction-delta}, \ref{p1}, one can shown that for any fixed $l\in\CC'$,
$$ \nu(z) = \lim_n \frac{[\LL(z)']^n l}{\langle [\LL(z)']^nl,1\rangle}. $$
\end{proof}

Since $\lambda(z)$ does not vanish for $|z|<\delta_0$, there exists a unique holomorphic function $P$ defined
for $|z|<\delta_0$, such that $P(0)=0$ and
$$ e^{P(z)} = \frac{\lambda(z)}{\lambda(0)}. $$
We have $P'(0) = 0$, $P''(0) = \sigma^2$ (where $\sigma^2$ is defined by (\ref{eq-sigma})).
In general, we have $P'(z) = \langle \nu(z), f h(z)\rangle$, and thus $P'(0) = \esp[f] = 0$.
This and the fact that $P''(0)$ equals the limit in (\ref{eq-sigma}) are classical calculations, see \cite{CP90},
see also Remark \ref{rem-sigma2}.
 
We turn now to estimate the Fourier transform. We have for any $z\in\C$,
$$ \esp[\exp(z S_n f)]  
= \langle \nu(0), \frac{\LL(0)^n}{\lambda(0)^n} e^{z S_n f} h(0) \rangle
= \langle \nu(0), \frac{\LL(z)^n}{\lambda(0)^n} h(0) \rangle. $$
Let $z$ be such that $|z|\leq\delta_0$. Therefore, we have
$\esp[\exp(z S_n f)] = \exp(nP(z)) \varphi_n(z)$ where
$$ \varphi_n(z) = \langle \nu(0), \frac{\LL(z)^n}{\lambda(z)^n} h(0) \rangle. $$
Now, we observe that
$$  \varphi_n(z) 
= \frac{\langle \nu(z), h(0) \rangle}{\langle \nu(z), u_n(z)\rangle}
\:\:\textrm{ with }\:\:
u_n(z) = \dfrac{ \LL(z)^n h(0) }{\langle\nu(0),\LL(z)^n h(0)\rangle}\in\CC. $$
Since $\nu(z)\in\CC'$, by (\ref{eq-ECxy}), we have
$\varphi_n(z) \in E_\CC(u_n(z), h(0))$. 
We have $\langle \nu(0), h(0)\rangle
=\langle \nu(0), u_n(z)\rangle = 1$ so $1\in E_\CC(u_n(z), h(0))$ and thus
$$ \exp(-\delta_\CC(u_n(z),h(0)))\leq |\varphi_n(z)|\leq \exp(\delta_\CC(u_n(z),h(0))). $$
Using Lemmas \ref{lemma-diameter} and \ref{lemma-51}, we have
\begin{eqnarray}
 \delta_\CC(h(0),u_n(z)) &\leq& \delta_\CC(\LL(0)h(0),\LL(0)u_{n-1}(z))\nonumber\\&&+
\delta_\CC(\LL(0) u_{n-1}(z), \LL(z) u_{n-1}(z)) 
\leq 2D_\R+\Delta_0. \label{eq-CapDelta}
\end{eqnarray}
\begin{remark}
Instead of (\ref{eq-CapDelta}), it is tempting to do the following reasoning.
Define $\eta=\tanh(3D_\R/4)$, then we have also by Theorem \ref{th-contraction-delta}
\begin{eqnarray*}
\delta_\CC(h(0), u_n(z)) &\leq &\eta\delta_\CC(h(0), u_{n-1}(z))+\Delta_0\\
&\leq& \frac{1 -\eta^n}{1-\eta} \Delta_0 
\leq\frac{ \Delta_0}{1-\eta}  = \frac{e^{3D_\R/2} +1}{2} \Delta_0.
\end{eqnarray*}
This estimate is optimal only if $\Delta_0$ and $D_\R$ are both quite small. Anyway, for our choice
of $\delta_0$ (which gives $\Delta_0\approx 4.65$), (\ref{eq-CapDelta}) is always better.
\end{remark}
Define $\Delta = 2D_\R +\Delta_0$.
Then for all $n\geq 1$ and $z\in\C$, $|z|\leq\delta_0$, we have
$$\exp(-\Delta) \leq |\varphi_n(z)|\leq \exp(\Delta). $$
Therefore, $\varphi_n$ is a holomorphic function such that $\varphi_n(0)=1$ 
and which maps the open disk $\{z:|z|<\delta_0\}$ into the open annulus
$A(\Delta) = \{\zeta: e^{-\Delta}< |\zeta|< e^\Delta\}$.
Furthermore, differentiation of $\varphi_n$ leads to $\varphi_n'(0)  = \esp[S_n f] - n P'(0) = 0$.

\begin{lemma} \label{lem-ring}
Let $\varphi$ be a holomorphic function from the open disk $\{z:|z|<\delta_0\}$ to the annulus
$A(\Delta)$ such that $\varphi(0)=1$, and $\varphi'(0)=0$. Let $0<\alpha<1$. Then for all
$z\in\C$ with $|z|\leq \alpha \delta_0$, we have
$$ |\varphi(z)-1| \leq \frac{\exp(C(\alpha) \Delta) -1}{\alpha^2\delta_0^2} |z|^2,
\textrm{ where }C(\alpha)= \frac{2}{\pi}\log\frac{1+\alpha}{1-\alpha}.$$
\end{lemma}

\begin{proof}
We can write $\varphi=\exp(\psi)$ where $\psi$ is a holomorphic function with values into the vertical
strip $V(\Delta) = \{\zeta:-\Delta<\Re(\zeta)<+\Delta\}$, and such that $\psi(0)=0$.
The Poincaré distance to $0$ (cf. e.g. \cite{Mil06}) of
$V(\Delta)$ is given by
$$ d_{V(\Delta)}(z,0) = \log \frac{|e^{i\pi z/(2\Delta)} +1| + |e^{i\pi z/(2\Delta)}-1|}
{|e^{i\pi z/(2\Delta)} +1| -|e^{i\pi z/(2\Delta)}-1|}. $$
Writing $z=x+iy$, this gives
$$ \cosh(d_{V(\Delta)}(0, x+iy)) = \frac{\cosh(\pi y /(2\Delta))}
{\cos(\pi x / (2\Delta))}. $$
For any $x\in (0, \pi/2)$ and $y\in\R$, we have
$\cos(x)\cosh(|x+iy|) \leq \cosh(y)$. 
Therefore, we have for any $z\in V(\Delta)$,
$$ |z|\leq \frac{2\Delta}{\pi} d_{V(\Delta)}(0,z). $$
Since holomorphic functions are contraction for the Poincaré metric, we have for $|z|\leq \alpha\delta_0$
\begin{eqnarray} \label{eq-ca}
|\psi(z)| &\leq& \frac{2\Delta}{\pi} d_{V(\Delta)}(\psi(0),\psi(z))
\leq \frac{2\Delta}{\pi} d_{D_{\delta_0}}(0,z) \nonumber\\
&=& \frac{2\Delta}{\pi} \log \frac{\delta_0 +|z|}{\delta_0-|z|}
\leq C(\alpha)\Delta.
\end{eqnarray}
For any $z\in\C$, we have $|e^z-1|\leq e^{|z|}-1$. 
Combining this with (\ref{eq-ca}), the maximum principle yields
$$ \sup_{|z|\leq \alpha\delta_0} \left|\frac{\varphi(z) -1}{z^2} \right|\leq
\frac{ e^{C(\alpha)\Delta}-1}{\alpha^2\delta_0^2}. $$
Hence the result. 
\end{proof}

\begin{remark} \label{rem-sigma2}
Further differentiation of $\varphi_n$ shows that
$\varphi_n''(0) = \esp[(S_n f)^2] - n P''(0)$. So, using Lemma \ref{lem-ring}, we see that
$$ |\esp[(S_n f)^2] - n P''(0)|\leq 2 \frac{e^{C(\alpha)\Delta}-1}{\alpha^2 \delta_0^2}. $$
This gives a quite good constant for the convergence rate in (\ref{eq-sigma}) (but the convergence 
rate in $1/n$ might be obtained directly).
\end{remark}


Now we estimate the pressure function.

\begin{lemma} \label{lemma-pression}
For all $|z|< \delta_0$,
$$ \Re(P(z)) \leq |\Re(z)| \|f\|_\infty< \delta_0\|f\|_\infty. $$
\end{lemma}

\begin{proof}
First, observe that $\LL(z)$ is a bounded linear operator when acting on $C(X;\C)$ (the space of complex
valued bounded continuous functions on $X$ endowed with $\|\cdot\|_\infty$) with spectral radius
$r_\infty(\LL(z))$. Since $\LL(z) h(z) = \lambda(z) h(z)$ with $h(z)\in\lipXC\subset C(X;\C)$, we have
$|\lambda(z)| \leq r_\infty(\LL(z))$. Now, for any $n\geq 1$ and $u\in C(X;\C)$,
\begin{eqnarray}
\|\LL(z)^n u \|_\infty
&=& \| \LL(0)^n e^{z S_n f} u \|_\infty
\leq \| \LL(0)^n\|_\infty \| e^{z S_n f}\|_\infty \|u\|_\infty\nonumber\\
&\leq& \| \LL(0)^n\|_\infty \exp( n|\Re(z)| \|f\|_\infty ) \|u\|_\infty.\label{eq-sprad}
\end{eqnarray}
Since $\LL(0)^n$ is a positive operator, its norm is attained at $1$. So we have
$$\|\LL(0)^n\|_\infty = \|\LL(0)^n 1\|_\infty
\leq \|\LL(0)^n 1\|_{\lipXC} \leq \|\LL(0)^n \|_{\lipXC}. $$
Therefore, $r_\infty(\LL(0))\leq r_{\lipXC} (\LL(0))= \lambda(0)$. Reporting in (\ref{eq-sprad}), we see that
$$|\lambda(z)| \leq r_\infty(\LL(z))\leq \lambda(0) \exp(  |\Re(z)| \|f\|_\infty).$$
\end{proof}

\begin{lemma} \label{lemma-pression2}
We have the following inequalities.
$$ \sigma^2 \leq \frac{4\|f\|_\infty}{\delta_0},
\textrm{ and }
|P'''(0)| \leq \frac{36\|f\|_\infty}{\delta_0^2}.$$
Let $\alpha$ such that $0<\alpha<1$. Then for any $z$ such that $|z|\leq\alpha\delta_0$, we have
$$ \left|P(z)-\frac{\sigma^2z^2}{2}\right|
\leq \frac{6\|f\|_\infty|z|^3}{\delta_0^2(1-\alpha^3)}. $$
$$ \left|P(z)-\frac{\sigma^2 z^2}{2} - \frac{P'''(0) z^3}{6}\right|
\leq \frac{18\|f\|_\infty |z|^4}{\delta_0^3(1-\alpha^4)}. $$
\end{lemma}

\begin{proof}
Let $\beta>0$, and let $Q$ be a holomorphic function from the open disk $\{|z|<\delta_0\}$ to the
left half-plane $\{\Re(w)<\beta\}$. Let $k\geq 1$, and assume that $Q^{(j)}(0) = 0$ for all $0\leq j< k$.
Define the Möbius transformation $R$ by
$$ R(w) = \frac{w}{2\beta - w}. $$ Then $R$ maps the left half-plane $\{\Re(w)<\beta\}$
conformally onto the open unit disc $\{|\zeta|<1\}$. One checks that $(R\circ Q)^{(j)} (0)=0$ for $j< k$ and
that $(R\circ Q)^{(k)}(0) = R'(0) Q^{(k)}(0) = Q^{(k)}(0)/ (2\beta)$. Therefore,
$ z^{-k} R\circ Q(z) $ defines a holomorphic function from the disk $\{|z|<\delta_0\}$ to the unit disk, and the
maximum principle yields
$$ \forall z\textrm{ s.t. } |z|<\delta_0,\quad |R\circ Q(z)|\leq \frac{|z|^k}{\delta_0^k}
\quad \textrm{and}\quad \frac{|Q^{(k)}(0)|}{(2\beta) k!}\leq \frac{1}{\delta_0^k}. $$ The Möbius transformation
$R^{-1}$ maps the closed disk $\{\zeta:|\zeta|\leq r \}$ 
(to which $Q(z)$ belongs from the preceding inequality when $r=(\delta_0^{-1}|z|)^k$)
onto the closed disk of diameter
$[R^{-1}(-r), R^{-1}(+r)]$. We have $R^{-1}(\zeta) = 2\beta\zeta /(\zeta+1)$ so that
$$\forall z\textrm{ s.t. } |z|<\delta_0,\quad
|Q(z)| \leq R^{-1}\left(-\frac{|z|^k}{\delta_0^k}\right)=\frac{ 2\beta |z|^k}{\delta_0^k - |z|^k}. $$
We obtain the desired inequalities setting $Q(z)=P(z)$, $k=2$, $\beta=\|f\|_\infty\delta_0$, then
$Q(z)=P(z)-\sigma^2 z^2/2$, $k=3$, $\beta=\|f\|_\infty\delta_0 + \sigma^2\delta_0^2/2\leq 3\|f\|_\infty\delta_0$
 and so on.
\end{proof}

Finally, we prove Theorem \ref{main-thm}.
From Lemma \ref{lemma-pression2}, we have
$$ \alpha := \frac{\delta_0\sigma^2}{25\|f\|_\infty} \leq \frac{4}{25}. $$
One checks that one has
$$ \frac{6\|f\|_\infty \alpha}{\delta_0\sigma^2 (1-\alpha^3)} \leq \frac{1}{4}. $$
Moreover, the constant $C(4/25)$ of Lemma \ref{lem-ring} is not greater than $2/9$.
Let $t\in\R$,
$|t|\leq \alpha\sigma\delta_0\sqrt{n}$, $t\neq 0$. 
With our choice of $\alpha$, we get from Lemma \ref{lemma-pression2}
\begin{eqnarray*}
\left|P\left(\frac{it}{\sigma\sqrt{n}}\right)+\frac{t^2}{2n}\right|
&\leq& \frac{6\|f\|_\infty |t|^3}{\delta_0^2 \sigma^3 (1-\alpha^3) n\sqrt{n}} \leq \frac{t^2}{4n}.
\end{eqnarray*}
So $\Re(P(it/(\sigma\sqrt{n}))) \leq -t^2/(4n)$.
From Lemmas \ref{lem-ring} and \ref{lemma-pression2}, we get
\begin{eqnarray*}
\frac{1}{|t|}\left|\esp\big[\exp(\frac{itS_n f}{\sigma\sqrt{n}})\big] - e^{-t^2/2}\right| &=&
\frac{1}{|t|}\left|e^{nP(\frac{it}{\sigma\sqrt{n} })}(\varphi_n(\frac{it}{\sigma\sqrt{n} })-1) 
+ e^{nP(\frac{it}{\sigma\sqrt{n} })}-e^{-t^2/2}\right| \\
&\leq& e^{-t^2/4} \frac{e^{2\Delta/9}-1}{\alpha^2\delta_0^2\sigma^2 n}|t|
+ e^{-t^2/4}\frac{6\|f\|_\infty t^2}{\delta_0^2 \sigma^3(1-\alpha^3)\sqrt{n}}\\
&\leq& e^{-t^2/4}\left(
\frac{e^{D_\R/2}e^{2\Delta_0/9}-1}{\alpha\delta_0\sigma\sqrt{n}}+
\frac{t^2}{4\alpha\delta_0\sigma\sqrt{n}}\right).
\end{eqnarray*}
So finally, we have
\begin{equation}\label{eq-1}
\int_{-\alpha\delta_0\sigma\sqrt{n}}^{\alpha\delta_0\sigma\sqrt{n}}
\frac{1}{|t|}\left|\esp\big[\exp(\frac{itS_n f}{\sigma\sqrt{n}})\big] - e^{-t^2/2}\right|dt
\leq \frac{2\sqrt{\pi}e^{D_\R/2}e^{2\Delta_0/9} - \sqrt{\pi}}{\alpha\delta_0\sigma\sqrt{n}}.
\end{equation}
We now use the following classic inequality which is established in \cite{Fe71}:
\begin{equation} \label{eq-feller}
 \left|\mu( (-\infty,x] ) -\frac{1}{\sqrt{2\pi}}\int_{-\infty}^x e^{-t^2/2} dt\right|
\leq \frac{1}{\pi}\int_{-T}^{+T} \left|\frac{\hat{\mu}(t) - e^{-t^2/2}}{t}\right|dt
 + \frac{24}{\pi T\sqrt{2\pi}},
\end{equation}
where $\mu$ is any probability measure on $\R$ with $0$ mean,
$x\in\R$ and $T>0$ are arbitrary. Letting $T=\alpha\delta_0
\sigma\sqrt{n}$, using (\ref{eq-1}) and the fact that $2\sqrt{\pi}e^{2\Delta_0/9}\leq 10$
\begin{eqnarray*}
\left|P\left(\frac{S_n f}{\sigma\sqrt{n}}\leq x\right)  -\frac{1}{\sqrt{2\pi}}\int_{-\infty}^x e^{-t^2/2} dt\right|
&\leq& \frac{10e^{D_\R/2} + 8 }{\pi\alpha\delta_0\sigma\sqrt{n}}
\leq \frac{40\cosh^2(D_\R/4)}{\pi\alpha\delta_0\sigma\sqrt{n}}\\
&\leq&\frac{11460 \cosh^6(D_\R/4)\|f\|_\infty(\|f\|_\infty +|f|_\ell)^2}{\sigma^3\sqrt{n}}.
\end{eqnarray*}

\begin{remark} \label{remark-PC90}
It is possible to refine the estimate (\ref{eq-1}), although the constant becomes more complicated.
Proceeding as in the proof of Theorem 1 of \cite{CP90}, we define
$$z= n P\left(\frac{it}{\sigma\sqrt{n}}\right) +\frac{t^2}{2} -ib,
\quad b=-\frac{P'''(0)t^3}{6\sigma^3 \sqrt{n}}. $$ Since $b$ is real, we have
$|e^{z+ib}-(1+ib)|\leq |z|e^{|z|} +b^2/2$. By Lemma \ref{lemma-pression2}, for $|t|\leq \alpha\delta_0\sigma
\sqrt{n}$,
$$ |z|\leq \frac{18 \|f\|_\infty t^4}{\sigma^4 n\delta_0^3(1-\alpha^4)}
\leq \frac{18\|f\|_\infty \alpha^2 t^2}{\sigma^2\delta_0(1-\alpha^4)}
=\frac{18 \alpha t^2}{25(1-\alpha^4)} \leq \frac{t^2}{8}. $$
So we have
\begin{align*}
\left|e^{nP(\frac{it}{\sigma\sqrt{n}})} - e^{-t^2/2}\left(1-\frac{it^3 P'''(0)}{6\sigma^3\sqrt{n}}\right) \right|
& \leq e^{-t^2/2}\left(e^{t^2/8}\frac{18\|f\|_\infty t^4}{\sigma^4 n\delta_0^3(1-\alpha^4)}
+\frac{|P'''(0)|^2 t^6}{72 \sigma^6 n} \right) \\
&\leq e^{-t^2/2}\left(e^{t^2/8}\frac{18\|f\|_\infty t^4}{\sigma^4 n\delta_0^3(1-\alpha^4)}
+\frac{18\|f\|_\infty^2 t^6}{\delta_0^4 \sigma^6 n}\right).
\end{align*}
So we get
\begin{multline*}
\int_{-\alpha\delta_0\sigma\sqrt{n}}^{\alpha\delta_0\sigma\sqrt{n}}
\frac{1}{|t|}\left|\esp\big[\exp(\frac{itS_n f}{\sigma\sqrt{n}})\big] - e^{-t^2/2}
\left(1-\frac{it^3 P'''(0)}{6\sigma^3\sqrt{n}}\right)\right|dt \\
\leq \frac{4(e^{D_\R/2}e^{2\Delta_0/9}-1)}{\alpha^2\delta_0^2\sigma^2 n}
+\frac{129 \|f\|_\infty}{\sigma^4\delta_0^3 n}
+\frac{288\|f\|_\infty^2}{\delta_0^4\sigma^6 n}.
\end{multline*}
\end{remark}

\section{Non-Markov maps} \label{appendix-nonmarkov}

Our argument works as soon as we can find a real Birkhoff cone which is a strict contraction
for the Hilbert metric, and for which one has an explicit estimate of the contraction rate.
It is thus possible to extend our method to non-markov piecewise expanding maps on the interval
using the ideas of \cite{Liv95bis}.

Consider a map $T$ from $[0,1]$ into itself, and assume the following. There exists a finite subdivision
$0=a_0<\dots<a_p=1$ such that the restriction of $T$ to the open interval $(a_{i-1},a_i)$ can be extended
to a $C^2$ map on $[a_{i-1},a_i]$. Assume that
$$ \inf |T'|\geq \gamma > 2. $$
Denote $\A_0$ the partition (up to a finite number of point) $((a_{i-1},a_i))_i$ and
$\A_n=\A_0\vee T^{-1} \A_0 \vee \dots\vee T^{-n}\A_0$. Denote also by $\LL:L^1\to L^1$ the transfer operator associated
to $T$ by (\ref{eq-transfer-dual}). Since we no longer assume that $T (a_{i-1},a_i) = (0,1)$, the
space of continuous functions is not stable by $\LL$ in general. Here, the natural space is the space $BV([0,1])$ of bounded variations functions on $[0,1]$.
Since $\LL$ acts naturally on $[0,1]$, we consider the space $BV$ as a subspace of $L^1$.
Recall that if $V(f)$ is the total variation of the function $f:[0,1]\to\C$, then
the total variation $v(f)$ of the a.e.-class of $f$ is $$ v(f) = \inf\{V(g):\: g=f\ a.e.\}. $$ One has also for instance $v(f)=V(f_0)$ where $f_0$ is the unique
function which is right continuous on $[0,1)$ and left continuous at $1$ and such that $f_0=f$ a.e. The space $BV$ is endowed with the norm
$\|f\|_{BV}=v(f)+\|f\|_1$. For $f\in BV$, we have $\|f\|_\infty \leq \|f\|_{BV}$.

 Under these conditions, the following inequality due to Lasota and Yorke (\cite{LY}) holds for all $g\in BV$
\begin{equation} \label{eq-lasota-yorke}
 v(\LL g) \leq \frac{2}{\gamma} v(g) + A \int_0^1 |g(x)|dx,
\end{equation}
where $$A=\sup \dfrac{|T''|}{|T'|^2} + \dfrac{2}{\inf_{I\in \A_0} |I| \inf_I |T'|}.$$
Note that $1\geq |T(I)| \geq |I| \inf_I |T'|$ so $A\geq 1$. Recall that iterations of (\ref{eq-lasota-yorke}) lead to
\begin{equation} \label{eq-lasota-yorke-iterate}
 v(\LL^n g) \leq \left(\frac{2}{\gamma}\right)^n v(g) + A\frac{1-(2\gamma^{-1})^n}{1-2\gamma^{-1}} \int |g(x)|dx.
\end{equation}

We now recall Liverani's result (\cite{Liv95bis}).  Assume the following ``covering'' property: for all $n$,
there exists $N(n)$ such that for all $I\in \A_n$,
\begin{equation} \label{eq-covering}
 T^{N(n)} I = [0,1],
\end{equation}
where equality has to be understood up to a finite number of points.
Define the cone
\begin{equation}
 \CC_\R = \left\{g \in BV:\: g\geq 0 \textrm{ and } v(g) \leq a \int_0^1 g \right\},
\end{equation}
where\footnote{The choice of $a$ is also a bit arbitrary, any $a>A(1-2\gamma^{-1})^{-1}$ would do.} $$a= \frac{2A}{1-2\gamma^{-1}}. $$
Then for the transfer operator $\LL=\LL(0)$ satisfies
$\LL(\CC_\R\MZ)\subset \CC_\R\MZ$; and there exists $N^*\geq 1$ such that the Hilbert diameter of
$\LL^{N^*} \CC_\R$ is not greater than $D_\R<\infty$. The quantities $N^*$ and $D_\R$ can be explicited
in terms of $T$ and some $N(n_0)$, where $n_0$ is some integer depending on $\gamma$ and $a$.
Liverani gave abstract conditions that insure the existence of $N(n)$, namely that the invariant measure
$h_0 dm$ given by Lasota and Yorke's theorem (\cite{LY}) is mixing and satisfies $\inf h_0 >0$. However,
the situation is not as simple as in the Markov setting because
there is no general bound on $N(n_0)$. So in practice, one has to find the value of $N(n_0)$ ``by hand''.
We refer\footnote{To avoid confusion, we have kept the notation $D_\R$ for the Hilbert
diameter of $\LL(0)$, but it is denoted by $\Delta$ in \cite{Liv95bis}. Our $\gamma$ is denoted by $\lambda$ in \cite{Liv95bis}.}
 to \cite{Liv95bis}, Appendix I for discussion on this matter and for the formulas for $N^*$, $D_\R$, $n_0$.

Now, we need an analogue of Lemma \ref{lemma-epsilonz}. We fix an observable $f\in BV$, and we consider the complex operator
$\LL(z) u = \LL e^{zf} u = \LL(0) e^{zf} u$.
\begin{lemma}
 Let $n\geq 1$ and $z\in\C$. Then for any $m\in\CC_\R'$ and $u\in\CC_\R$, we have
$$ |\langle m, \LL(z)^n u \rangle - \langle m, \LL^n u\rangle | \leq e^{n|\Re(z)|\|f\|_\infty} |z| M_n(f) \langle m,\LL^n u\rangle.$$
The quantity $M_n(f)$ is defined by
$$ M_n(f)= \frac{5}{1-(2\gamma^{-1})^n} \Big( n\|f\|_\infty + (2\gamma^{-1})^n (\sharp \A_0)^n v(f)\Big). $$
\end{lemma}

\begin{proof}
Since the total variation $v(\cdot)$ is a seminorm, we have
$v(g)=\sup\langle l,g\rangle$ where the supremum is taken over all $l\in BV([0,1],\R)'$ such that 
$\langle l,w\rangle \leq v(w)$ for all $w$.
So the cone $\CC_\R$ is generated by the real functionals $g\in BV([0,1];\R) \mapsto g(x)$ (where $g$ is taken to be right continuous on $[0,1)$ and left
continuous at $1$) and the family of functionals $$\langle m,g\rangle = a\int g - \langle l,g\rangle,$$ where
$\langle l,w\rangle \leq v(w)$ for all $w$. For such a functional $m$ and $u\in\CC_\R$, we have by (\ref{eq-lasota-yorke-iterate})
\begin{align}
 \langle m,\LL^n u\rangle =a\int  u - \langle l, \LL^n u\rangle &\geq a\int u - v(\LL^n u) \nonumber\\
&\geq (1-(2\gamma^{-1})^n)\left(a - \frac{A}{1-2\gamma^{-1}}\right)\int u. \label{eq-espilonz-nonmarkov0}
\end{align}
So we have for any $z\in\C$,
\begin{align}
 \big|\langle m,&\LL(z)^n u \rangle - \langle m, \LL^n u\rangle \big|
= \left| a\int \left(e^{z S_n f}-1\right)u  - \langle l ,\LL^n \left(e^{z S_n f}-1\right)u \rangle\right| \nonumber\\
&\leq a \int \left|e^{z S_n f}-1\right|u + v\left(\LL^n [(e^{z S_n f}-1)u] \right) \nonumber\\
&\leq \left(a +\frac{A}{1-2\gamma^{-1}}\right)e^{n|\Re(z)|\|f\|_\infty} n|z| \|f\|_\infty \int u
+(2\gamma^{-1})^n v\left(( e^{z S_n f} -1)u\right). \label{eq-epsilonz-nonmarkov1}
\end{align}
We have $v(e^{zS_nf} -1) \|u\|_\infty \leq |z| v(S_nf) \big[\exp(n|\Re(z)|\|f\|_\infty)\big]\big[v(u) +\int u\big]$ and $v(u)\leq a\int u$. Besides,
\begin{align*} 
v(S_nf) &\leq \sum_{k=0}^{n-1} v(f\circ T^k) = \sum_{k=0}^{n-1} \sum_{I\in \A_{k-1}} v_{T^k(I)}(f)\\
&\leq  \sum_{k=0}^{n-1} (\sharp \A_{k-1}) v(f) = \frac{(\sharp \A_0)^n -1}{\sharp \A_0 -1} v(f)
\leq (\sharp \A_0)^n v(f).
\end{align*}
So, combining with (\ref{eq-espilonz-nonmarkov0}) and (\ref{eq-epsilonz-nonmarkov1}), we get
\begin{align*}
 \big|\langle &m,\LL(z)^n u - \LL^n u\rangle \big| \\
&\leq |z|e^{n|\Re(z)|\|f\|_\infty}  \left( n\|f\|_\infty \left(2a +\frac{A}{1-2\gamma^{-1}}\right)
+(a+1)(2\gamma^{-1})^n (\sharp \A_0)^n v(f) \right) \int u \\
&\leq |z|e^{n|\Re(z)|\|f\|_\infty} \frac{5}{1-(2\gamma^{-1})^n} \left(n\|f\|_\infty +(2\gamma^{-1})^n (\sharp \A_0)^n v(f) \right)
\langle m,\LL^n u\rangle.
\end{align*}
Finally, the triangular inequality yields
\begin{align*}
 |\LL^n [ (e^{zS_n f}-1) u] (x)| &\leq \LL^n [ |e^{zS_nf }-1| u ] (x) \leq e^{n|\Re(z)|\|f\|_\infty}n|z|\|f\|_\infty \LL^n u(x) \\
&\leq e^{n|\Re(z)|\|f\|_\infty} |z| M_n(f) \LL^n u(x).
\end{align*}
\end{proof}

We now indicate where we have to modify the proof of Theorem \ref{main-thm}. First, the cone $\CC_\R$ has a non-empty interior in $BV$: the constant function
$1$ is in its interior. It is also of bounded aperture: if $g\in\CC_\R$, then $\|g\|_{BV}\leq (a+1)\int g$. The same is true for the canonical
complexification $\CC_\C$ of $\CC_\R$ (\cite{Rugh07}, Proposition 5.4). We still note $\CC=\CC_\C\MZ$.

Then, we replace (\ref{eq-delta0}) by the following choice for $\delta_0$
$$ \delta_0 \max_{N^*\leq n < 2N^*} M_n(f) =  \frac{1}{3(1+\cosh(D_\R/2))}. $$ This choice implies that
$\exp(n \delta_0 \|f\|_\infty) \leq e^{1/30}$ for $N^*\leq n<2N^*$.
We prove as in Lemma \ref{lemma-diameter} that for $|z|\leq \delta_0$ and $N^*\leq n <2N^*$,
$ \LL(z)^n \CC \subset \CC $ and 
\begin{equation} \label{eq-nonmarkov3} \sup_{u\in\CC} \delta_\CC(\LL(z)^nu,\LL(0)^n u) \leq 3.51 =: \Delta_0. \end{equation}
From this, we deduce that $\LL(z)^n \CC\subset \CC$ and (\ref{eq-nonmarkov3}) still hold for any $n\geq N^*$.

Corollary \ref{lemma-eigen} still holds without any change. Indeed, using the complex contraction, one first
proves that the operator $\LL(z)^{N^*}$ has a spectral gap with left and right
eigenvectors $\nu(z)\in\CC'$ and $h(z)\in\CC$ suitably normalized and depending analytically on $z$. Since the leading eigenvalue of $\LL(z)^{N^*}$ is
simple, the operator $\LL(z)$ must also have a spectral gap with the same left and right eigenvectors.

Regarding Equation (\ref{eq-CapDelta}), we notice first that $u_n(z)\in\CC$ for $n\geq N^*$. Therefore, for $n\geq 2N^*$ we have
\begin{align*}
 \delta_\CC(h(0),u_n(z)) &\leq \delta_\CC(\LL(0)^{N^*}h(0), \LL(0)^{N^*} u_{n-N^*}(z))\\
&\quad +\delta_\CC(\LL(0)^{N^*} u_{n-N^*}(z),\LL(z)^{N^*} u_{n-N^*}(z)) \leq 2D_\R +\Delta_0.
\end{align*}

Lemma \ref{lemma-pression} is still valid, however, one has in the proof to replace the space $C(X;\C)$ by $L^\infty(X;\C)$ and the space
$\lipXC$ by $BV(X;\C)$.

The end of the proof goes in the same way, and we conclude that for all $n\geq 2N^*$,
$$ \left|P\left(\frac{S_n f}{\sigma\sqrt{n}}\leq x\right)  -\frac{1}{\sqrt{2\pi}}\int_{-\infty}^x e^{-t^2/2} dt\right|
\leq\frac{9168 \cosh^6(D_\R/4)\|f\|_\infty(M_{N^*}(f))^2}{\sigma^3\sqrt{n}}.
$$

\appendix
\section{Linear convexity}\label{app-linconv}

\begin{remark}
 The condition of Proposition \ref{prop-linearlyconvex} is actually also necessary. This condition is always satisfied
if $V_{\R}$ is a separable Banach space. On the contrary,
there exists nonseparable real Banach spaces and real convex cones for which such an $m$ does not exist,
see \cite{Kr48}.
\end{remark}

\begin{proof}[Proof of proposition \ref{prop-linearlyconvex}]
Define
$$ \mathcal{S} = \{ f\in V_{\C}': \: \forall x,\,y\in\mathcal{C}_{\R},\,x,\,y\textrm{ independent,}\,
\Re\left(\langle f,x\rangle\overline{\langle f,y\rangle}\right)>0 \}. $$
We first prove that $\mathcal{S}=\mathcal{C}'$.
Indeed, let $f\in\mathcal{S}$, and $z\in\mathcal{C}\MZ$. Write $z=\lambda(x+iy)$, $\lambda\in\C^*$,
$x$, $y\in\mathcal{C}_{\R}$. Since for any $u\in\mathcal{C}_{\R}\MZ$, we may find $v\in\mathcal{C}_{\R}$
independent of $u$, we have $\Re(\overline{\langle f,u\rangle}\langle f,v\rangle)>0$, hence
$\langle f,u\rangle \neq 0$. Now, if $y=0$ then $\lambda^{-1} z = x\in\mathcal{C}_{\R}\MZ$
so $\langle f,z\rangle \neq 0$. If $x=\alpha y$, then $0\neq \lambda^{-1} z = (\alpha + i)y$ and
$y\in\mathcal{C}_{\R}\MZ$ so $\langle f,z\rangle \neq 0$. Finally, if $x$ and $y$ are independent
then $\Re(\overline{\langle f,y \rangle}\langle f,x\rangle)>0$. Write $\langle f,x\rangle = 
re^{i\alpha}$, $\langle f,y\rangle=se^{i\beta}$, $r$, $s>0$ and $\alpha-\beta\in(-\pi/2,\pi/2)$.
Then $\Re(\langle f,\lambda^{-1}e^{-i\beta}z\rangle) =r\cos(\alpha-\beta) >0$. Thus $\langle f,z
\rangle \neq 0$.

Now, suppose that $f\notin\mathcal{S}$. Then, there exists independent $x$, $y\in\mathcal{C}_{\R}$
such that $\Re(\overline{\langle f,y \rangle}\langle f,x\rangle)\leq 0$. We may suppose that
$\langle f,x\rangle \neq 0$ and $\langle f,y\rangle\neq 0$. Write again
$\langle f,x\rangle = re^{i\alpha}$, $\langle f,y\rangle=se^{i\beta}$, $r$, $s>0$ and
$\theta= \pi+\alpha-\beta\in[-\pi/2,\pi/2]$. Then $\langle f, z\rangle = 0$, where
$z=s^{-1}re^{i\theta}y+x$. We have $z\neq 0$ because $x$ and $y$ are independent. Finally,
if $\theta\in[0,\pi/2]$ then $z\in\mathcal{C}_{\R} +i\mathcal{C}_{\R}$, and if
$\theta\in[-\pi/2,0]$ then $iz\in\mathcal{C}_{\R} +i\mathcal{C}_{\R}$.

We prove now
$$ x\in\mathcal{C}\MZ \quad\iff\quad \forall f\in\mathcal{S}, \langle f,x\rangle\neq 0.$$
Recall that we assume the existence of $m\in\mathcal{C}_{\R'}$ positive on $\mathcal{C}_{\R}\MZ$.
Then $\langle m,z\rangle \neq 0$ for all $z\in\mathcal{C}\MZ$. Let $x\in V_{\C}$ and suppose
that $\langle f,x\rangle\neq 0$ for all $f\in\mathcal{S}$. Let $l_1$, $l_2\in\mathcal{C}'_{\R}$. One
checks easily that $m+(l_1+il_2)$ belongs to $\mathcal{S}\cap(\mathcal{C}'_{\R} + i\mathcal{C}'_{\R})$.
Therefore, $\langle l_1+il_2 ,x\rangle \neq -\langle m,x\rangle$. Define
$K = \{ \langle l_1+il_2,x\rangle:\:l_1,l_2\in\mathcal{C}'_{\R}\}.$ $K\subset \C$ is a convex subcone
of $\C$, and $-\langle m,x\rangle \notin K$ so that $K\neq \C$.
Let again $l_1$, $l_2\in\mathcal{C}'_{\R}$ and suppose that
$\Re(\langle l_1,x\rangle\overline{\langle l_2,x\rangle}) <0$. We write
$\langle l_1,x\rangle = re^{i\alpha}$, $\langle l_2,x\rangle = se^{i\beta}$, $r$, $s>0$ and
$\theta = \pi+\alpha - \beta \in (-\pi/2,\pi/2)$. If $\theta\geq 0$, then for any $\delta\geq 0$,
small enough, $\delta+\theta\in[0,\pi/2)$. Hence,
$f=s^{-1}e^{i(\delta+\theta)}l_2 \in \mathcal{C}'_{\R}+i\mathcal{C}'_{\R}$ and
$g=r^{-1}e^{i\delta}l_1 \in \mathcal{C}'_{\R}+i\mathcal{C}'_{\R}$. But
$\langle g,x\rangle =-\langle f,x\rangle = e^{i(\alpha+\delta)}$. It means that for $\delta\geq 0$
small enough, we have $\pm e^{i(\alpha+\delta)} \in K$ and thus, $K=\C$ which is impossible.
Similarly, if $\theta\in (-\pi/2,0]$, we consider
$f=r^{-1}e^{(i\delta-\theta)}l_1$, $g=s^{-1}e^{i\delta}l_2$ and we prove that
for any $\delta\geq 0$, small enough, $\pm e^{i(\beta+\delta)}\in K$. This is also impossible.
\end{proof}

\end{document}